\documentclass[a4paper,twoside,10pt]{article}
\usepackage{ucs}
\usepackage[utf8x]{inputenc}
\usepackage[T1]{fontenc}
\usepackage[english]{babel}
\usepackage{subfigure}
\usepackage{amsmath, amsthm, amsfonts, amssymb}
\usepackage{latexsym}
\usepackage{color}
\usepackage{stmaryrd}
\usepackage{endnotes}
\usepackage{wrapfig}
\usepackage{fancyhdr,lastpage}
\usepackage{ccaption}
\usepackage{enumitem}
\usepackage{xspace}
\usepackage{ifthen}
\usepackage{xkeyval}
\usepackage{remreset}
\usepackage{xcolor}
\usepackage[marginpar]{todo}
\usepackage{calc}
\usepackage{multirow}
\usepackage{graphicx}

\makeatletter

\fancypagestyle{plain}{%
	\fancyhf{} %
	\fancyfoot{}%
}

\let\orig@enumerate =\enumerate
\renewenvironment{enumerate}{\orig@enumerate [label=(\roman{*}), ref=(\roman{*})]}
{\endlist}

\definecolor{grispuce}{gray}{0.7}

\newcommand{\resp}{respectively \xspace}
\newcommand{\oas}{$\omega$-almost surely ($\omega$-as)\xspace \renewcommand{\oas}{$\omega$-as\xspace}}
\newcommand{\oeb}{$\omega$-essentially bounded  ($\omega$-eb)\xspace \renewcommand{\oeb}{$\omega$-eb\xspace}}

\newtheorem{defi}{Definition}[section]
\newtheorem{lemm}[defi]{Lemma}
\newtheorem{prop}[defi]{Proposition}
\newtheorem{theo}[defi]{Theorem}
\newtheorem{coro}[defi]{Corollary}

\@addtoreset{defi}{section}
\renewcommand{\thesection}{%
	\arabic{section}%
}
\renewcommand{\thesubsection}{%
	\ifnum \value{section}>0
		\thesection.%
	\else%
	\fi%
	\arabic{subsection}%
}
\renewcommand{\thedefi}{%
	\ifnum \value{section}>0
		\thesection.%
	\else%
	\fi%
	\arabic{defi}%
}

\newcommand{\rem}{\paragraph{Remark :}}
\newcommand{\rems}{\paragraph{Remarks :}}

\newcommand{\nota}{\paragraph{Notation :}}

\newcommand{\N}{\mathbf{N}}

\newcommand{\Q}{\mathbf{Q}}
\newcommand{\R}{\mathbf{R}}
\newcommand{\C}{\mathbf{C}}
\renewcommand{\H}{\mathbf{H}}

\newcommand{\set}[2]{%
	\left\{\vphantom{#2}#1\ \right| \left. \vphantom{#1}#2\right\}%
}

\define@choicekey*{planhyp}{corps}[\val \nr]{r,c,h,}[]{
	\ifcase\nr 
		\def\ph@corps{\R}
	\or
		\def\ph@corps{\C}
	\or
		\def\ph@corps{\H}
	\or
		\def\ph@corps{}
	\fi
}
\presetkeys{planhyp}{corps}{}
\newcommand{\HP}[2][]{
	\begingroup
		\setkeys{planhyp}{#1}
		\ifthenelse{\equal{\ph@corps}{}}
			{\mathbf{H}_{#2}}
			{\mathbf{H}_{#2}\left(\ph@corps \right)}
	\endgroup
}
\newcommand{\gro}[3]{\left< #1,#2\right>_{#3}}
\newcommand*{\dist}[3][]{
	\ifthenelse{\equal{#1}{}}
		{\left| #2- #3 \right|}
		{
			\ifthenelse{\equal{#1}{SC}}
			{\left\| #2- #3 \right\|}
			{\left| #2- #3 \right|_{#1}}
		}
}
\newcommand*{\distV}[1][]{
	\ifthenelse{\equal{#1}{}}
		{\left| \ . \ \right|}
		{
			\ifthenelse{\equal{#1}{SC}}
			{\left\| \ . \ \right\|}
			{\left| \ . \ \right|_{#1}}
		}
}

\newcommand*{\geo}[3][]{
	\ifthenelse{\equal{#1}{}}
		{\left[ #2, #3 \right]}
		{\left[ #2, #3 \right]_{#1}}
}

\newcommand*{\triang}[4][]{
	\ifthenelse{\equal{#1}{}}
		{\left[ #2, #3, #4 \right]}
		{\left[ #2, #3, #4 \right]_{#1}}
}

\newcommand{\ball}[1]{
	B\left( #1\right)
}

\newcommand{\ann}[2]{
	A\left( #1,#2\right)
}

\define@boolkey{longueurtrans}{stable}[true]{}
\define@key{longueurtrans}{espace}[]{\def \lt@espace{#1}}
\presetkeys{longueurtrans}{stable=false,espace}{}
\newcommand*{\len}[2][]{
	\begingroup
		\setkeys{longueurtrans}{#1}
		\ifKV@longueurtrans@stable {
			\ifthenelse{\equal{\lt@espace}{}}
				{\left[ {#2}\right]^{\infty}}
				{\left[ {#2}\right]^{\infty}_{\lt@espace}}
		}
		\else {
			\ifthenelse{\equal{\lt@espace}{}}
				{\left[ {#2}\right]}
				{\left[ {#2}\right]_{\lt@espace}}
		}
		\fi
	\endgroup
}
\newcommand{\rips}[3][]{
	\ifthenelse{\equal{#1}{}}
		{P_{#2}\left( #3\right)}
		{P_{#2}^{(#1)}\left( #3\right)}
}
\newcommand{\card}[1]{
	\left| #1\right|
}

\newcommand{\sdp}[3][]{
	\ifthenelse{\equal{#1}{}}
		{#2 \rtimes #3}
		{#2 \rtimes _{#1} #3}
}

\newcommand{\free}[1]{\mathbf F_{#1}}

\newcommand{\burn}[2]{\mathbf B _{#1}(#2)}
\newcommand{\fantomB}{\vphantom{\Big\vert}\!}
\renewcommand{\epsilon}{\varepsilon}
\renewcommand{\phi}{\varphi}
\renewcommand{\leq}{\leqslant}
\renewcommand{\geq}{\geqslant}

\newcommand{\makebiblio} {
	\bibliography{/Users/coulonr/Maths/Articles/bibliography} 
	\bibliographystyle{abbrv} 
}

\makeatother
\begin{document}

\title{Growth of periodic quotients of hyperbolic groups}
\author{R\'emi Coulon}

\maketitle

\begin{abstract}
	Let $G$ be a non-elementary torsion-free hyperbolic group. 
	We prove that the exponential growth rate of the periodic quotient $G/G^n$ tends to the one of $G$ as $n$ odd approaches infinity.
	Moreover we provide an estimate at which the convergence is taking place.
\end{abstract}
	
\tableofcontents

\section*{Introduction}

\paragraph{} A group $G$ is said to have finite exponent $n$ if for every $g \in G$, $g^n = 1$.
At the beginning of the 20\textsuperscript{th} century, W.~Burnside settled the following problem (now known as the \emph{Bounded Burnside Problem}) \cite{Bur02}.
Is  a finitely generated group with finite exponent necessary finite?
In order to study this question the natural object to look at is the free Burnside group of rank $k$ and exponent $n$ denoted by $\burn kn$.
It is the quotient of the free group of rank $k$ denoted by $\free k$ by the (normal) subgroup $\free k^n$ generated by the $n$-th power of all elements of $\free k$.
It is the largest group of rank $k$ and exponent $n$.

\paragraph{} For a long time it was only established that $\burn kn$ was finite for some small exponents ($n=2$ \cite{Bur02}, $n=3$ \cite{Bur02,LevWae33}, $n=4$ \cite{San40} and $n=6$ \cite{Hal57}). 
The finiteness of $\burn 25$ is still open.
In 1968, P.S.~Novikov and S.I.~Adian achieved a breakthrough.
In a series of three articles \cite{NovAdj68c}, they provided the first examples of infinite free Burnside groups. 
More precisely they proved the following result.
If $k \geq 2$ and $n$ is an odd integer larger than 4381, then $\burn kn$ is infinite.
Their result has been improved in many directions.
In particular S.V.~Ivanov \cite{Iva94} and I.G.~Lysenok \cite{Lys96} solved the case of even exponents.
Since free Burnside groups of sufficiently large exponents are infinite a natural question is how ``big'' they are.
This can be measured by the exponential growth rate.

\paragraph{} Given a finitely generated group $G$ endowed with the word metric with respect to some finite generating set of $G$, its \emph{(exponential) growth rate} is defined to be
\begin{displaymath}
	\lambda = \lim_{r \rightarrow \infty} \sqrt[r]{\card{\ball r}},
\end{displaymath}
where $\card{\ball r}$ denotes the cardinal of the ball of radius $r$ of $G$.
If $\lambda>1$ one says that $G$ has \emph{exponential growth}.
($\lambda$ depends on the generating set, however having exponential growth is a property of the group $G$).
Furthermore if for every generating set the corresponding growth rate is uniformly bounded away from 1 then $G$ has \emph{uniform exponential growth}.

\paragraph{} In his book \cite{Adi79}, S.I.~Adian proved that free Burnside groups of sufficiently large odd exponents are not only infinite but also exponentially growing.
Latter D.~Osin showed that they are uniformly non-amenable, and therefore they have uniform exponential growth \cite{Osi07}.
An other approach can be found in \cite{Atabekyan:2009ih}.

\paragraph{} In 1991, using a diagrammatical description of graded  small cancellation theory, A.Y.~Ol'shanski\u\i\ proved an analogue for hyperbolic groups of the Novikov-Adian Theorem \cite{Olc91}.
\begin{theo}[Ol'shanski\u\i\,{\cite{Olc91}}]
\label{res: Olc periodic quotient hyp groups}
	Let $G$ be a non-elementary, torsion-free hyperbolic group.
	There exists a critical exponent $n(G)$ such that for all odd integers $n \geq n(G)$, the quotient $G/G^n$ is infinite.
\end{theo}
Non-elementary hyperbolic groups are known to have uniform exponential growth \cite{Kou98}.
On the other hand hyperbolic groups are growth tight \cite{ArzLys02}.
This means that, given such a group $G$ and a finite generating set $A$, for any infinite normal subgroup $N$ of $G$, the exponential 
growth rate of $G/N$ with respect to the natural image of $A$ is strictly less 
than the exponential growth rate of $G$ with respect to $A$.
Therefore we were wondering what the growth rate of the periodic quotients $G/G^n$ could be.
In particular is there a gap between the respective growth rates of $G$ and $G/G^n$?
The following theorem answers this question negatively: the growth rate of $G/G^n$ converges to the one of $G$ as $n$ odd approaches infinity.
Moreover we provide an estimate for the rate at which this convergence is taking place.

\begin{theo}
\label{res: main theorem}
	Let $G$ be a non-elementary, torsion-free hyperbolic group and $\lambda$ its exponential growth rate with respect to a finite generating set $A$ of $G$.
	There exists a positive number $\kappa$ such that for sufficiently large odd exponents $n$ the exponential growth rate of $G/G^n$ with respect to the image of $A$ is larger than 
	\begin{displaymath}
		\lambda \left(1- \frac \kappa n\right).
	\end{displaymath}

\end{theo}

In the case of free Burnside groups we even have a much more accurate estimate.

\begin{theo}
\label{res: main theorem - free}
	Let $k \geq 2$.
	Let $A$ be a free generating set of $\free k$ (i.e. with exactly $k$ elements).
	There exists a positive number $\kappa$ such that for sufficiently large odd exponents $n$ the exponential growth rate of $\burn kn$ with respect of the image of $A$  is lager than 
	\begin{displaymath}
		(2k-1) \left(1- \frac \kappa{(2k-1)^{n/2}}\right).
	\end{displaymath}
\end{theo}

\paragraph{} Our proof extends the ideas of S.I.~Adian.
However considering hyperbolic groups instead of free groups makes it much more complicated and requires new tools.
Let us first recall the key argument of Adian's approach. 

\paragraph{Main fact.} Let $v$ be a reduced word representing an element of $\free k$.
If $v$ does not contain a subword of the form $w^{16}$, then $v$ induces a non-trivial element of $\burn kn$ for every odd integer $n \geq 665$.

\paragraph{}In particular, two distinct reduced words which do not contain a 8-th power induce different elements of $\burn kn$.
Therefore, it is sufficient to estimate the growth rate of $F_8$, the set of reduced words without 8-th power.
This is done by induction on the length of the words.
The main steps of this proof are recalled in Section~\ref{sec:avoiding large powers - free case}.

\paragraph{}Consider now an arbitrary non-elementary, torsion-free, hyperbolic group $G$ endowed with the word metric $\distV$.
Following A.Y.~Ol'shanski\u\i, a $(L,m)$-power is an element of $G$ that can be written $uw^nu'$ where $u$ and $u'$ have length at most $L$.
An element $g$ of $G$ is $(L,m)$-aperiodic if it can not be written $g= g_1g_2g_3$ where $|g| = |g_1|+ |g_2|+|g_3|$ and $g_2$ is a $(L,m)$-power.
The proof of Theorem~\ref{res: Olc periodic quotient hyp groups} relies on the following fact.
In \cite{Olc91}, A.Y.~Ol'shanski\u\i\ proved  the existence of constants $L$, $\epsilon$ and $n(G)$, which only depends on $G$ with the following property.
Let $n$ be an odd integer larger than $n(G)$, let $m \leq \epsilon n$.
Then the set of $(L,m)$-aperiodic elements embeds into $G/G^n$.
The infiniteness of $G/G^n$ follows from the one of $(L,m)$-aperiodic elements.
An other approach based on techniques developed by T.~Delzant and M.~Gromov \cite{DelGro08} can be found in \cite{Cou11b}.

\paragraph{}
Hence one way to prove Theorem~\ref{res: main theorem} is to compute the growth rate of the set of $(L,m)$-aperiodic elements.
Instead of reasoning with words we consider geodesic path the Cayley graph $X$ of $G$.
However this definition of $(L,m)$-aperiodic words does not behave well with the operations of extending geodesics or taking subgeodesics.
For instance we would like to have the following fact.
Let $g$ and $g'$ be two element of $G$ such that $g$ lies on a geodesic between 1 and $g$.
If $g$ is $(L,m)$-aperiodic but $g'$ is not, then the $m$-th power in $g'$ can be read ``at the end'' of the geodesic representing $g'$.
Nevertheless, since $X$ is not uniquely geodesic, this statement does not hold: the element $g$ could contain a  $(L+ 4 \delta,m)$-power as illustrated on Figure~\ref{fig: powers pb 1}.

\begin{figure}[ht]
\centering
	\includegraphics{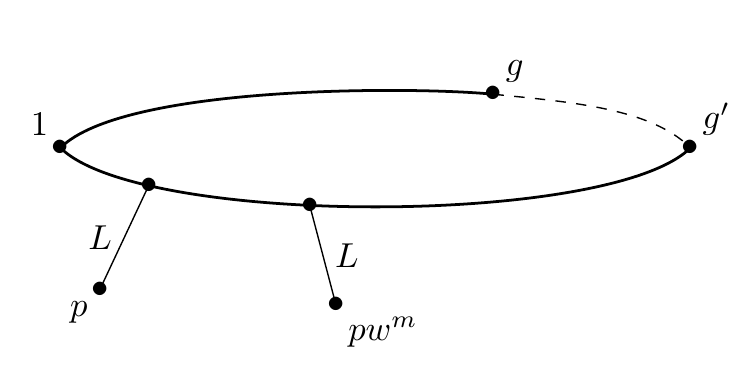}
\caption{Extending aperiodic elements.}
\label{fig: powers pb 1}
\end{figure}

\paragraph{}
To avoid this difficulty we focus on a particular set of geodesics.
We fix an arbitrary order on the generating set $A\cup A^{-1}$.
Thus the set of geodesics inherits of the lexicographical order.
For every $g \in G$, $\sigma_g$ is the smallest (for the lexicographical order) geodesic joining 1 to $g$.
We call such a path a \emph{lexicographic geodesic}.
In particular, if $g' \in G$ lies on $\sigma_g$ then $\sigma_{g'}$ is the subpath of $\sigma_g$ between 1 and $g'$.
Then we adopt the following definition.
An element $g \in G$ contains a $(L,m)$-power if there are $p \in G$ and a non-trivial cyclically reduced element $w \in G$ such that both $p$ and $pw^m$ belong to the $L$-neighborhood of $\sigma_g$.

\paragraph{}This adaptation leads to an other difficulty.
Given an element $g \in G$, we need to be sure that $\sigma_g$ can be extended in ``sufficiently many ways'' in a lexicographic geodesic.
However this could be impossible (see Fig.~\ref{fig: powers pb 2}).
This question is handled in Section~\ref{sec: growth of cone types}.
For every $r$ we construct a subset $F$ of $G$ which, among others, satisfies the following property.
For all $g \in F$ the number of elements $g' \in F$ such that $\sigma_{g'}$ extends $\sigma_g$ by a length $r$ is larger that $\nu \lambda^r$, where $\nu$ is some constant which only depends on $G$ and $A$ and $\lambda$ is the exponential growth rate of $G$.
Our proof uses as a tool the Canon cone types \cite{Can84}.

\begin{figure}[ht]
\centering
	\includegraphics{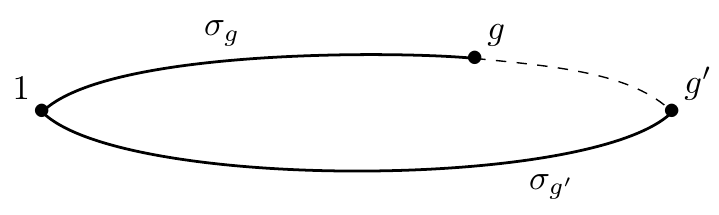}
\caption{Extending lexicographic geodesics.}
\label{fig: powers pb 2}
\end{figure}

\paragraph{}
Finally we prove that the set of $(L,m)$-aperiodic elements of $F$ grows exponentially with a rate larger than $\lambda(1-\kappa/n)$ (see Section~\ref{sec:avoiding large powers}).
Our theorem follows then from Ol'shanski\u\i's work (see Section~\ref{sec: periodic quotients}).

\paragraph{Acknowledgment.}
Part of this work was done during my stay at the \emph{Max-Planck-Institut f\"ur Mathematik}, Bonn, Germany.
I would like to express my gratitude to all faculty and staff from the MPIM for their support and warm hospitality.
Many thanks also goes to A.Y.~Ol'shanski\u\i\ for related conversations.

\section{Hyperbolic geometry}
\label{sec: recall hyperbolic spaces}

\paragraph{}
In this section we fix notations and review some of the standard facts on hyperbolic spaces and hyperbolic groups (in the sense of Gromov).
For more details we refer the reader to the original paper of M.~Gromov \cite{Gro87} or to \cite{CooDelPap90,GhyHar90}.
\paragraph{}
Let $G$ be a group generated by a finite set $A$.
We denote by $X$ the Cayley graph of $G$ with respect to $A$.
The vertices of $X$ are the elements of $G$.
For every $g\in G$ and $a \in A\cup A^{-1}$, $g$ is joined to $ga$ by an edge labeled by $a$.
The group $G$ acts on the left by isometries on $X$.

\paragraph{}
Given two points $x,x' \in X$, $\dist x{x'}$ stands for the distance between them.
The Gromov product of three points $x,y,z \in X$ is defined by
\begin{displaymath}
	\gro xyz = \frac 12 \left\{ \fantomB \dist xz + \dist yz - \dist xy \right\}.
\end{displaymath}
The space $X$ is said to be $\delta$-hyperbolic if for all $x,y,z,t \in X$, 
\begin{displaymath}
	\gro xyt \geq \min \left\{ \gro xzt, \gro zyt \right\} - \delta.
\end{displaymath}

\rem
The constant $\delta$ depends on $A$. 
Nevertheless for a group, being hyperbolic (for some $\delta$) does not depend on the generating set.
In this article we fix once for all the generating set $A$.
Therefore the hyperbolicity constant $\delta$ of $X$ is fixed as well.
Without loss of generality we can assume that $\delta \geq 1$.
More generally, greek letters will represent constants which only depend on $G$ and $A$.
Moreover in the rest of the article we assume that $G$ is torsion-free and non-elementary i.e., non virtually cyclic.

\paragraph{}
As a consequence of hyperbolicity, the geodesic triangles of $X$ are $4\delta$-thin i.e., for every $x,y,z \in X$, for every $p$ (\resp $q$) lying on a geodesic between $x$ and $y$ (\resp between $x$ and $z$), if $\dist xp = \dist xq \leq \gro yzx$ then $\dist pq \leq 4\delta$.

\paragraph{}
Let $g \in G$.
For simplicity of notation, $|g|$ stands for the distance $\dist g1$.
This is exactly the word length of $g$ with respect to $A$.
To measure the action of $g$ on $X$ we define two quantities:
the \emph{translation length} $\len g$ and the \emph{stable translation length} $\len[stable]g$.
\begin{displaymath}
\len g = \inf_{x \in X} \dist {gx}x; \quad \len[stable] g = \lim_{n \rightarrow \infty} \frac 1n \dist{g^nx}x
\end{displaymath}
They are related by the following inequality: $\len[stable] g \leq \len g \leq \len[stable] g + 50 \delta$.
A \emph{cyclically reduced isometry} is an element $g \in G$ such that $\len g = |g|$.
Every conjugacy class of $G$ contains such an isometry.
The set of all non-trivial cyclically reduced isometries is denoted by $C$.

\paragraph{}
If $\len[stable] g >0$ then $g$ is called \emph{hyperbolic}.
It is known that every element of $G$ is either hyperbolic or has  finite order (in the latter case it is said to be \emph{elliptic}).
For our purpose we assumed that $G$ was torsion-free.
Therefore every non-trivial isometry is hyperbolic.
\paragraph{}
In a hyperbolic group, the range of stable translation lengths is discrete:
\begin{theo}[Delzant, {\cite{Del96}}]
\label{res: stable translation length discrete}
	There exists a constant $\tau \in \Q_+^*$ such that for all $g \in G$, $\len[stable]g \in \tau\N$.
\end{theo}
In particular, for every $g \in G$, hyperbolic we have $\len[stable]g \geq \tau$.

\paragraph{}
Given $r \in \R_+$ and $g \in G$, we denote by $B(g,r)$ the close ball of $G$ of center $g$ and radius $r$ i.e., the set of $h\in G$ such that $\dist gh \leq r$. 
If $g$ is the trivial element we simply write $\ball r$.
For all $e \in \R_+$ the \emph{annulus} $\ann re$ is the set of elements $g \in G$ such that $r-e \leq |g| \leq r$.
If $r\geq 0$ and $e \geq 1$, then $\ann re$ is not empty.

\paragraph{}
If $P$ is a finite subset of $G$, $\card P$ stands for its cardinal.
In order to estimate the size of an infinite subset of $G$ we use the exponential growth rate

\begin{defi}
\label{def: growth}
	Let $P$ be a subset of $G$.
	The \emph{(exponential) growth rate} of $P$ is the quantity
	\begin{displaymath}
		\limsup_{r\rightarrow+\infty} \sqrt[r]{\card{P \cap \ball r}}
	\end{displaymath}
\end{defi}

We denote by $\lambda$ the growth rate of $G$.
Since the map $r \rightarrow \card{\ball r}$ is submultiplicative, $\lambda$ satisfies in fact
\begin{displaymath}
	\lambda = \lim_{r \rightarrow +\infty} \sqrt[r]{\card{\ball r}} = \inf_{r \in \N^*} \sqrt[r]{\card{\ball r}}.
\end{displaymath}
In particular for all $r \in \N$, $\card {\ball r} \geq \lambda^r$.
The next proposition gives an upper bound for $\card {\ball r}$.

\begin{prop}[Coornaert, {\cite{Coo93}}]
\label{res:size of balls}
	There exists $\alpha \geq 1$ which only depends on $G$ and $A$ such that for all $r \in \R_+$,  $\card{\ball r} \leq \alpha \lambda^r$.
\end{prop}

\section{Growth of cone types}
\label{sec: growth of cone types}

\subsection{Essential cone types}

\begin{defi}
\label{def: cone type}
	Let $g\in G$.
	The \emph{cone type} of $g$ is the set of elements $u \in G$ such that there exists a geodesic of $X$ between 1 and $gu$ that passes through $g$.
	We denote it by $T_g$.
\end{defi}

Since $G$ is a hyperbolic group, the set of all cone types, denoted by $\mathcal T$, is finite \cite[Chap. 12, Th. 3.2]{CooDelPap90}.

\begin{defi}
\label{def:essential type}
	We say that a cone type $T \in \mathcal T$ is \emph{essential} if 
	\begin{displaymath}
		\exists a >0, \; \forall r \geq 0, \; \exists s\geq r, \quad \card {T\cap\ball s} \geq a \lambda^s,
	\end{displaymath}
	where $\lambda$ is the growth rate of $G$.
\end{defi}

\nota An element $g \in G$ is \emph{essential} if its cone type $T_g$ is essential. 
The set of all essential elements is denoted by $E$.
We write $\mathcal T_E$ for the set of all essential cone types.

\rem It follows easily from the definition that the growth rate of an essential cone type is exactly $\lambda$.
Roughly speaking, the essential elements are the ones who are responsible for the growth of $G$.

\begin{prop}
\label{res:strong definition essential type}
	There exists $\beta > 0$ (which only depends on $G$ and $A$) such that for all $T \in \mathcal T_E$, for all $r \in \R_+$, 
	\begin{displaymath}
		\card{T\cap \ball r} \geq \beta \lambda^r.
	\end{displaymath}
\end{prop}

\begin{proof}
	Note that $\mathcal T_E$ is finite.
	Hence it is sufficient to prove the following statement:
	\begin{displaymath}
		\forall T \in \mathcal T_E,\; \exists \beta >0,\; \forall r \geq 0,\; \card{T\cap \ball r} \geq \beta \lambda^r.
	\end{displaymath}
	Let $T$ be an essential type.
	By definition there exists $a > 0$ such that for all $r \geq 0$ there is $s\geq r$ such that $\card{T\cap \ball s}\geq a \lambda^s$.
	Let $r\in \R_+$.
	We denote by $s$ a real number larger than $r$ such that $\card{T\cap \ball s}\geq a \lambda^s$.
	Every element of $T \cap \ball s$ can be written $uv$ where $u \in T\cap \ball r$ and $v \in \ball{s-r+1}$.
	(Note that $r$ is not necessary an integer, therefore $v$ is a ball of radius $s-r+1$ and not $s-r$.)
	Consequently $\card {T\cap \ball s} \leq \card{T \cap \ball r}\card {\ball {s-r+1}}$.
	Using Proposition~\ref{res:size of balls} we obtain
	\begin{displaymath}
		a \lambda^s \leq \card {T\cap \ball s} \leq \alpha \lambda^{s-r+1}\card{T \cap \ball r}.
	\end{displaymath}
	Thus for all $r \geq 0$, $\card{T \cap \ball r} \geq  \alpha^{-1}\lambda^{-1}a \lambda^r$.
\end{proof}

\begin{lemm}
\label{res: gu essential gives g essential}
	Let $g \in G$.
	Let $u \in T_g$.
	If $gu$ is essential then so is $g$.
\end{lemm}

\begin{proof}
	By definition of the cone type, $uT_{gu}$ is a subset of $T_g$.
	Hence for all $r \geq |u|$,
	\begin{math}
		 \card{T_g \cap \ball r} \geq  \card{T_{gu}\cap\ball{r-|u|}}.
	\end{math}
	However $gu$ is essential.
	It follows from Proposition~\ref{res:strong definition essential type} that for all $r \geq |u|$,
	\begin{displaymath}
		\card{T_g \cap \ball r}
		\geq \card{T_{gu}\cap\ball{r-|u|}}
		\geq \beta \lambda^{-|u|} \lambda^r.
	\end{displaymath}
	Thus $g$ is essential.
\end{proof}

Let $g \in G$ and $u \in T_g$.
According to the previous lemma, if $gu$ is essential, so is $g$.
The converse statement is not necessary true.
Nevertheless, the next proposition gives a lower bound for $\card{T_g\cap g^{-1}E\cap \ann re}$ which is the number of elements $u \in T_g \cap \ann re$ such that $gu$ is essential. 
(Recall that $\ann re$ is the annulus of radius $r$ defined in Section~\ref{sec: recall hyperbolic spaces}.)

\begin{prop}
\label{res: growth set of essential elements}
	There exists $\gamma >0$ (which only depends on $G$ and $A$) such that for all $g \in E$ and for all $r \geq 0$, for all $e \geq 1$.  
	\begin{displaymath}
		\card{T_g\cap g^{-1}E\cap \ann re} \geq \gamma \lambda^r.
	\end{displaymath}
\end{prop}

\begin{proof}
	Note that it is sufficient to prove the proposition for $e=1$.
	Let $\gamma >0$.
	Suppose the proposition were false.
	There exists an essential element $g \in E$ and $r \in \R_+$ such that $\card{T_g\cap g^{-1}E\cap \ann r1} < \gamma \lambda^r$.
	Negating the definition of essential types, we have
	\begin{displaymath}
		\forall T \in \mathcal T \setminus \mathcal T_E,\; \exists s \geq 0,\; \forall t \geq s,\; \quad \card{T\cap \ball t} < \gamma \lambda^t.
	\end{displaymath}
	Recall that the set of cone types $\mathcal T$ is finite.
	Thus we have in fact
	\begin{equation}
	\label{eqn: growth set of essential elements}
		\exists s \geq 0,\; \forall T \in \mathcal T \setminus \mathcal T_E, \; \forall t \geq s,\; \quad \card{T\cap \ball t} < \gamma \lambda^t.
	\end{equation}
	Let $t \geq s$.
	It follows from the definition of cone types that
	\begin{displaymath}
		T_g\cap \ball {r+t} \subset T_g \cap \ball {r-1} \cup \left( \bigcup_{u \in T_g \cap \ann r1 } u\left(T_{gu}\cap \ball{t+1}\right)\right).
	\end{displaymath}
	Since $g$ is essential, Proposition~\ref{res:strong definition essential type} yields
	\begin{displaymath}
		\beta \lambda^{r+t}
		\leq \card{T_g\cap \ball {r+t}}
		\leq \card{T_g \cap \ball {r-1} } + \sum_{u \in T_g \cap  \ann r1 } \card{T_{gu}\cap \ball {t+1}}.
	\end{displaymath}
	Let $u \in T_g \cap \ann r1$.
	If $u$ does not belong to $g^{-1}E$, then $gu$ is not essential.
	By~(\ref{eqn: growth set of essential elements}), $\card{T_{gu}\cap \ball {t+1}} \leq \gamma \lambda^{t+1}$.
	On the other hand, if $u \in g^{-1}E$, then Proposition~\ref{res:size of balls} leads to $\card{T_{gu}\cap \ball {t+1}} \leq \card{\ball {t+1}} \leq \alpha \lambda^{t+1}$.
	It follows that $ \card{T_g\cap \ball {r+t}}$ is bounded above by
	\begin{displaymath}
		 \card{T_g \cap \ball {r-1} }
		+ \gamma \lambda^{t+1} \card{ T_g \cap \ann r1 \setminus g^{-1}E} + \alpha  \lambda^{t+1}\card{ T_g \cap g^{-1}E \cap \ann r1}.
	\end{displaymath}
	However by assumption, $\card{ T_g \cap g^{-1}E \cap \ann r1} \leq \gamma \lambda^r$.
	Moreover Proposition~\ref{res:size of balls} gives $\card{T_g \cap \ball {r-1}} \leq \alpha \lambda^{r-1}$ and $\card{ T_g \cap \ann r1 \setminus g^{-1}E} \leq \alpha \lambda^r$.
	Thus  for all $t \geq s$,
	\begin{math}
		\beta \lambda^{r+t}
		\leq \alpha \lambda^{r-1}  + 2\alpha\gamma \lambda^{r+t+1}.
	\end{math}
	Therefore $0<\beta \leq 2 \alpha \lambda \gamma$.
	This inequality holds for all $\gamma >0$, which is impossible.
\end{proof}

\subsection{Lexicographic types}

\paragraph{}Recall that $A$ is a finite generating set of $G$ and $X$ the Cayley graph of $G$ with respect to $A$.
In this section we fix an arbitrary order on $A\cup A^{-1}$.
Let $g,h \in G$.
Using the labeling of the edges of $X$, any geodesic joining $g$ to $h$ can be identify with a word over the alphabet $A \cup A^{-1}$ representing $hg^{-1}$.
Thus the set of geodesics inherits from the lexicographic order.
(We read the words from the left to the right.)
For all $g \in G$, we denote by $\sigma_g$ the geodesic joining 1 to $g$ which is the smallest for the lexicographic order.
We call it the \emph{lexicographic geodesic} from 1 to $g$.
Note that if $h \in G$ lies on $\sigma_g$ then $\sigma_h$ is exactly the subpath of $\sigma_g$ between 1 and $h$.

\begin{defi}
\label{def: lex type}
	Let $g\in G$.
	The \emph{lexicographic type} of $g$ is the set of elements $u \in G$ such that $\sigma_{gu}$ passes through $g$.
	We denote it by $L_g$.
\end{defi}

\rem It follows from the definition that $L_g$ is a subset of $T_g$.
Contrary to $\mathcal T$, it is not clear whether or not the set of all lexicographic types is finite.

\paragraph{}
Our goal is to construct a subset $F$ of $G$ such that its elements satisfy analogues for the lexicographic types of Proposition~\ref{res: growth set of essential elements} and Lemma~\ref{res: gu essential gives g essential}.
To that end, we need the following lemma.

\begin{lemm}[Arzhantseva-Lysenok,  {\cite[Lemma 5]{ArzLys02}}]
\label{res: separate elements}
	There exists a constant $\rho>0$ which only depends on $G$ and $A$ satisfying the following.
	For every finite subset $P$ of $G$ there is a subset $P'$ of $P$ such that $\card {P'}\geq \rho \card P$ and for all distinct $g,g' \in P'$, $\dist g{g'} > 20\delta$.
\end{lemm}

Recall that $\gamma$ is the constant given by Proposition~\ref{res: growth set of essential elements}.
Let us put $\nu = \rho\gamma \lambda^{-4\delta}\card{\ball {4\delta}}^{-1}$.
This number only depends on $G$ and $A$.
Let $r$ be a real number larger than $10 \delta$.
The set $F$ that we are going to build will depend on the parameter $r$, which represents a distance.
However for simplicity, we do not mention the dependence on $r$ in the notation.
First, we construct by induction a non-increasing sequence $(H_i)$ of subsets of $G$.
\begin{itemize}
	\item Put $H_0 = G$.
	\item Let $i \in \N$.
	Assume that $H_i$ is already defined.
	The set $H_{i+1}$ is 
	\begin{displaymath}
		H_{i+1} = \set{g \in H_i \fantomB}{\card{L_g\cap g^{-1}H_i \cap \ann r{9\delta}} \geq \nu \lambda^r}
	\end{displaymath}
\end{itemize}
The set $H$ is defined to be the intersection of all $H_i$'s.

\begin{lemm}
\label{res: weak definition of H_i}
	Let $g \in G$.
	Let $i \in \N$.
	If $\card{L_g\cap g^{-1}H_i \cap \ann r{9\delta}} \geq \nu \lambda^r$ then $g$ belongs to $H_{i+1}$.
\end{lemm}

\begin{proof}
	Let $j\leq i$.
	By construction $H_i$ is a subset of $H_j$.
	Therefore
	\begin{displaymath}
		\card{L_g\cap g^{-1}H_j \cap \ann r{9\delta}} \geq \card{L_g\cap g^{-1}H_i \cap \ann r{9\delta}} \geq \nu \lambda^r.
	\end{displaymath}
	A proof by induction on $j$ shows that for all $j \leq i+1$, $g$ belongs to $H_j$.
\end{proof}

\begin{prop}
\label{res: H_i is dense in E}
	For all $i \in \N$, for all $g \in E$, $B(g,4\delta)\cap H_i$ is non-empty.
\end{prop}

\begin{proof}
	We prove this proposition by induction on $i$.
	If $i=0$, the proposition follows from the fact that $H_0 = G$.
	Assume now that the proposition holds for $i \in \N$.
	Let $g \in E$.
	According to Proposition~\ref{res: growth set of essential elements}, $\card{T_g \cap g^{-1}E\cap \ann {r-4\delta}\delta} \geq \gamma \lambda^{r-4\delta}$.
	By Lemma~\ref{res: separate elements}, there exists a subset $P$ of $T_g \cap g^{-1}E\cap \ann {r-4\delta}\delta$ such that 	
	\begin{enumerate}
		\item $\card P \geq \rho \card{T_g \cap g^{-1}E\cap \ann {r-4\delta}\delta} \geq \rho\gamma \lambda^{r-4\delta}$,
		\item for all distinct $u,u' \in P$, $\dist u{u'} > 20 \delta$.
	\end{enumerate}
	Let $u \in P$.
	The isometry $gu$ is essential.
	Applying the induction assumption $B(gu, 4\delta) \cap H_i$ contains an element that we shall write $h_u=lv$ where $l$ belongs to $G$ with $|g|=|l|$ and $v$ to $L_l$.
	By construction $l$ (\resp $g$) belongs to a geodesic between 1 and $lv$ (\resp $gu$), thus
	\begin{displaymath}
		\dist{|u|}{|v|} = \dist{|gu|}{|lv|}\leq \dist{gu}{lv} \leq 4 \delta.
	\end{displaymath}
	Since $u$ lies in $\ann {r-4\delta}\delta$, $v$ is an element of $\ann r{9\delta}$.
	Moreover $\dist {gu}g \geq r - 5 \delta \geq \dist{lv}{gu}$.
	By hyperbolicity $\dist lg \leq 4 \delta$.
	Consequently $\left\{h_u | u \in P \right\}$ is a subset of  $\bigcup_{l \in B(g, 4\delta)} l\left(L_l\cap \ann r{9\delta}\right)$.
	On the other hand we claim that $\left\{h_u | u \in P \right\}$ is a subset of $H_i$ that contains exactly $\card P$ elements.
	Let $u,u' \in P$ such that $h_u = h_{u'}$.
	By triangle inequality we have
	\begin{displaymath}
		\dist u{u'} = \dist {gu}{gu'} \leq \dist{gu}{h_u} + \dist{h_{u'}}{gu'} \leq 8 \delta.
	\end{displaymath}
	By definition of $P$, $u = u'$.	
	Therefore $\bigcup_{l \in B(g, 4\delta)} l\left(L_l\cap \ann r{9\delta}\right)$ contains at least $\card P$ elements of $H_i$.
	Hence there is $l \in B(g, 4\delta)$ such that
	\begin{displaymath}
		\card {L_l\cap l^{-1}H_i\cap \ann r{9\delta}} \geq \card{\ball{4\delta}}^{-1} \card P \geq  \rho\gamma \card{\ball{4\delta}}^{-1}\lambda^{r-4\delta} = \nu \lambda^r .
	\end{displaymath}
	It follows from Lemma~\ref{res: weak definition of H_i}, that $l$ belongs to $B(g, 4\delta) \cap H_{i+1}$.
	Thus the proposition holds for $i+1$.
\end{proof}

\begin{coro}
\label{res: H is dense in E}
	For all $g \in E$, $B(g,4\delta)\cap H$ is non-empty.
\end{coro}

\begin{coro}
\label{res: relative growth H}
	For all $g \in H$, $\card{L_g \cap g^{-1}H \cap \ann r{9\delta}} \geq \nu \lambda^r$.
\end{coro}

\begin{proof}
	Both corollaries follow from the fact that $H  = \bigcap_{i \in \N}H_i$.
\end{proof}

\rem In particular the set $H$ is non empty.
Since $1$ is essential ($T_1 = G$) Corollary~\ref{res: H is dense in E} tells us that $\ball {4 \delta}$ contains an element of $H$.
Actually, the same proof as the one of Proposition~\ref{res: H_i is dense in E} shows that 1 belongs to $H$.

\paragraph{}Corollary~\ref{res: relative growth H} is an analogue for the lexicographic types of Proposition~\ref{res: growth set of essential elements}.
However given $g \in G$ and $u \in L_g$ such that $gu$ belongs to $H$ there is no reason that $g$ should also belong to $H$.
That is why we have to consider a subset $F$ of $H$ which will in addition satisfy an analogue of Lemma~\ref{res: gu essential gives g essential} (see Lemma~\ref{res: F stable under restriction}).
To that end we proceed by induction
\begin{itemize}
	\item Put $F_0 =\{1\}$.
	\item Let $i \in \N$ such that $F_i$ is already defined.
	The set $F_{i+1}$ is given by
	\begin{displaymath}
		F_{i+1} = \bigcup_{g\in F_i} g\left( L_g \cap g^{-1}H \cap \ann r{9\delta} \right).
	\end{displaymath}
\end{itemize}
Finally, the set $F$ is the union of all the $F_i$'s.
Note that $F$ is a subset of $H$.
Moreover for all $g \in F$, $L_g \cap g^{-1}H\cap \ann r{9\delta}$ lies inside $g^{-1}F$. 
Therefore Corollary~\ref{res: relative growth H} leads to the following result.

\begin{lemm}
\label{res: relative growth F}
	For all $g \in F$, $\card{L_g \cap g^{-1}F \cap \ann r{9\delta}} \geq \nu \lambda^r$.
\end{lemm}

\begin{lemm}
\label{res: F stable under restriction}
	Let $g \in F$ and $x$ be a point of $\sigma_g$.
	There exists $h \in F$ which lies on $\sigma_g$ between 1 and $x$ such that $\dist xh \leq r$.
\end{lemm}

\begin{proof}
	Let us denote by $i$, the smallest integer such that $x$ is on a geodesic $\sigma_l$ with $l \in F_i \cap \sigma_g$.
	If $i=0$ then $l=x=1$. Thus the lemma is obvious.
	Assume now that $i \geq 1$.
	By definition there exists $h \in F_{i-1}$ such that $l$ belongs to $h\left(L_h\cap h^{-1}H\cap \ann r{9\delta}\right)$.
	In particular $h$ and $l$ are two points of $\sigma_g$.
	We claim that $x$ lies on $\sigma_g$ between $h$ and $l$.
	Suppose it were false, then $x$ would lie on $\sigma_h$ which contradicts the minimality of $i$.
	However, by construction $\dist lh \leq r$
	Consequently $h$ is a point of  $F\cap \sigma_g$ between 1 and $x$ such that $\dist xh \leq \dist lh \leq r$.
\end{proof}

Finally we have proved the following result.
\begin{prop}
\label{res: all properties F}
	There is $\nu > 0$ (which only depends on $G$ and $A$) such that for all $r \geq 10 \delta$, there exists a subset $F$ of $G$ satisfying the following properties.
	\begin{enumerate}
		\item 1 belongs to $F$,
		\item for all $g \in F$, $\card{L_g \cap g^{-1}F \cap \ann r{9\delta}} \geq \nu \lambda^r$,
		\item for all $g \in F$ , for all $x\in \sigma_g$, there exists $h \in F$ which lies on $\sigma_g$ between 1 and $x$ such that $\dist xh \leq r$.
	\end{enumerate}
\end{prop}

\section{Avoiding large powers}
\label{sec:avoiding large powers}

The goal here is to estimate the growth rate of a subset of $G$ ``without'' power.
This section involves many parameters. 
As a warmup we start with the case of free groups.
We present briefly the idea used by S.I.~Adian in \cite{Adi79}.
The estimation that we obtain in that particular case will also be useful in Section~\ref{sec: periodic quotients}.

\subsection{The case of free groups}
\label{sec:avoiding large powers - free case}

In this section we assume that $A$ is a free generating set of $\free k$ i.e., it contains exactly $k$ elements.
Consequently the exponential growth rate of $\free k$ with respect to $A$ is $\lambda=2k-1$.
Let $m \in \N$.
An element of $\free k$ is said to be $m$-aperiodic, if the reduced word over the alphabet $A \cup A^{-1}$ representing it does not contain a subword of the form $u^m$.
We denote by $F_m$ the set of $m$-aperiodic elements of $\free k$

\begin{prop}
\label{res: min F_m with quantifiers - free case}
	For all integers $m \geq 2$, for all $s \in \N$,
	\begin{displaymath}
		\card{F_m\cap \ball{s+1}} \geq
		\lambda \card{F_m\cap \ball s} -  \frac {2k}{2k-1} \sum_{j\geq 1} \lambda^j\card{F_m\cap \ball{s+1-mj}}.
	\end{displaymath}
\end{prop}

\begin{proof}
	Let $s \in \N$.
	An $m$-aperiodic word $w$ of length $s+1$ can be written $w = w'a$ where $w'$ is an $m$-aperiodic word of length $s$ and $a$ an element of $A\cup A^{-1}$. 
	Since $w$ is reduced the number of possible choices for $a$ is $2k-1$.
	On the other hand, consider a reduced of the form $w'a$ where $w' \in F_m \cap B(s)$ and $a \in A\cup A^{-1}$.
	If such a word is not $m$ aperiodic, then there exists $j \in \N^*$, $u \in \free k$ with $|u|=j$ and $p \in F_m \cap \ball{s+1-mj}$ such that $w'a=pu^m$.
	The number of words of this last form is bounded above by 
	\begin{displaymath}
		\card{F_m\cap \ball{s+1-mj}}.\card{\ann j0} \leq 2k(2k-1)^{j-1}\card{F_m\cap \ball{s+1-mj}}.
	\end{displaymath}
	Therefore the number of reduced words of the form $w'a$ ($w' \in F_m \cap \ball s$ and $a \in A\cup A^{-1}$) which are $m$-aperiodic is bounded below by
	\begin{displaymath}
		(2k-1) \card{F_m\cap \ball s} -  2k \sum_{j\geq 1} (2k-1)^{j-1}\card{F_m\cap \ball{s+1-mj}},
	\end{displaymath}
	which gives the desired conclusion.
\end{proof}

\begin{prop}
\label{res: growth F_m - free case}
	Let $k \geq 2$.
	For every $a > 2k$, there exists an number $m_0$ such that for every integer $m \geq m_0$  the exponential growth rate of $F_m$ is larger than $\lambda\left(1 - a\lambda^{-m}\right)$.
\end{prop}

\begin{proof}
We consider the function $f_m : \left(\sqrt[m] \lambda, \lambda \right) \rightarrow \R$ defined by 
\begin{displaymath}
	f_m(\mu) = \lambda - \frac {2k\mu}{2k-1}\sum_{j \geq 1} \left(\frac \lambda {\mu^m}\right)^j = \lambda - \frac {2k\mu}{\mu^m - \lambda}
\end{displaymath}
Let $a > 2k$.
We put $\mu_m = \lambda\left(1 - a\lambda^{-m}\right)$.
The  sequence $\mu_m$ tends to $\lambda >1$ as $m$ approaches infinity.
Therefore
\begin{displaymath}
	f_m(\mu_m) = \lambda\left(1 - \frac {2k}{\lambda^m}\right) + \operatorname{o}\displaylimits_{m \rightarrow + \infty}(1).
\end{displaymath}
Since $a > 2k$ there exists a number $m_0$ such that for every integer $m \geq m_0$, $f_m(\mu_m) \geq \mu_m$.
Fix $m \geq m_0$.
For simplicity of notation we write $\mu$ for $\mu_m$.
We now prove by induction that for every $s \in \N$, $\card{F_m \cap \ball s} \geq \mu \card{F_m \cap \ball {s-1}}$.
The statement is true for $s=0$.
Assume that it holds for all integers less or equal to $s$.
In particular for every $j \geq 1$, $\card{F_m\cap \ball{s+1-mj}} \leq \mu^{1-mj}  \card{F_m \cap \ball s}$.
It follows then from Proposition~\ref{res: min F_m with quantifiers - free case} that 
\begin{displaymath}
	\card{F_m\cap \ball{s+1}} 
	\geq \left[\lambda - \frac {2k\mu}{(2k-1)} \sum_{j\geq 1}\left(\frac \lambda{\mu^m}\right)^j\right] \card{F_m\cap \ball s}.
\end{displaymath}
The expression between the brackets is exactly $f_m(\mu)$. 
Therefore the assumption holds for $s+1$.
A second induction proves that for every $s \in \N$, $\card{F_m \cap \ball s} \geq \mu^s$, which leads to the result.
\end{proof}

\subsection{The general case}
\label{sec:avoiding large powers - general case}

\paragraph{}
We now deal with the case of hyperbolic groups.
Recall that a cyclically reduced isometry is an element $g \in G$ such that $\len g= |g|$.
The set of non-trivial cyclically reduced elements is denoted by $C$ (see Section~\ref{sec: recall hyperbolic spaces}).
Let $L > 0$ and $m\in \N^*$.
Given $g \in G$, we say that $g$ contains a \textit{$(L,m)$-power}, if there exists $(l,w) \in G\times C$ such that both $l$ and $lw^m$ belong to the $L$-neighbourhood of $\sigma_g$.
If $g$ does not  contain any $(L,m)$-power it is called \textit{$(L,m)$-aperiodic}.

\rem Our definition of $(L,m)$-aperiodic elements is a slightly weaker form of the one of A.Y.~Ol'shanski\u\i\ \cite{Olc91}.
However it is sufficient to apply Ol'shanski\u\i's results (see the remark following Theorem~\ref{res: aperiodic one-to-one}).

\paragraph{}
Let $g \in G$ and $h$ be an element of $G$ which lies on $\sigma_g$.
By construction, the geodesic $\sigma_h$ is the subpath of $\sigma_g$ joining 1 and $h$. 
Therefore if $g$ is $(L,m)$-aperiodic, so is $h$.

\paragraph{}
Given a subset $F$ of $G$, we denote by $F_{L,m}$ the set of elements $g \in F$ which are $(L,m)$-aperiodic.
Our aim is to give a lower bound for the growth rate of $F_{L,m}$ for an appropriate subset $F$ of $G$.
More precisely we prove the following result.

\begin{prop}
\label{res: growth F_m}
	Let $L >0$. 
	There exist $a>0$ and $m_0 \in \N$ satisfying the following property.
	For all  integers $m \geq m_0$, there is a subset $F$ of $G$, such that the exponential growth rate of $F_{L,m}$ is larger than $\lambda \left(1-  a/m\right)$.
\end{prop}

\paragraph{}
The rest of this section is dedicated to the proof of Proposition~\ref{res: growth F_m}.
We first need to fix some parameters.
The constants $\tau$ and $\nu$ are the ones respectively given by Theorem~\ref{res: stable translation length discrete} and Proposition~\ref{res: all properties F}.
First we take $L >0$.
Its value will be made precise in the next section (see Theorem~\ref{res: aperiodic one-to-one}).
Let $r$ be a real number larger than $10 \delta$ and $m$ an integer such that $m \tau \geq 2 L + r$.
According to Proposition~\ref{res: all properties F}, there exists a subset $F$ of $G$ containing 1, such that 
\begin{enumerate}
	\item for all $g \in F$, $\card{L_g \cap g^{-1}F \cap \ann r{9\delta}} \geq \nu \lambda^r$,
	\item for all $g \in F$, for all $x\in \sigma_g$, there exists $h \in F$ which lies on $\sigma_g$ between 1 and $x$ such that $\dist xh \leq r$.
\end{enumerate}\
We now define auxiliary subsets of $G$.
\begin{itemize}
	\item $Z = \set{hu}{h \in F_{L,m}, u \in L_h \cap h^{-1}F \cap \ann r{9\delta}}$.
	(Note that $Z \subset F$.)
	\item For all $w \in G\setminus\{1\}$, $Z_w$ is the set of elements $g\in G$ that can be written $g=huw^mu'$ where
	\begin{enumerate}
		\item $h$ belongs to $F_{L,m}$,
		\item $|g| \geq |h| + |w^m| - 2L$.
		\item $|u| ,|u'|\leq L + r$.
	\end{enumerate}
\end{itemize}
Roughly speaking $Z$ denotes the set of elements of $F$ which geodesically extend a $(L,m)$-aperiodic isometry by a length $r$.
On the other hand, $Z_w$ contains the elements which extend a $(L,m)$-aperiodic isometry by a $m$-th power of $w$.
The idea of the next proposition is the following.
An $(L,m)$-aperiodic isometry of length $s+r$ can be obtained by extending a $(L,m)$-aperiodic isometry of length $s$.
However this extension should not involve a $m$-th power.

\begin{prop}
\label{res: Z-Zw subset F_n}
	The set $Z \setminus \bigcup_{w \in C} Z_w$ is contained in $F_{L,m}$.
\end{prop}

\begin{proof}
	Equivalently we prove that $Z \setminus F_{L,m} \subset \bigcup_{w \in C} Z_w$.
	Let $g$ be an element of $Z\setminus F_{L,m}$.
	In particular it belongs to $F\setminus F_{L,m}$.
	Therefore there exist $l \in G$ and $w \in C$ such that $l$ and $lw^m$ belong to the $L$-neighbourhood of $\sigma_g$.
	We denote by $p$ and $q$ respective projections of $l$ and $lw^m$ on $\sigma_g$.
	By replacing if necessary $w$ by $w^{-1}$ one can assume that $1$ $p$, $q$ and $g$ are ordered in this way on $\sigma_g$.
	
	\paragraph{Claim 1:} $\dist gq \leq r$.
	Assume on the contrary that this assertion is false.
	Since $g$ belongs to $Z$ there is $g' \in F_{L,m} \cap \sigma_g$ such that 
	\begin{displaymath}
		r- 9 \delta \leq \dist g{g'} \leq  r < \dist gq.
	\end{displaymath}
	It follows that $p$ and $q$ both belong to $\sigma_{g'}$.
	In particular $l$ and $lw^m$ lie in the $L$-neighbourhood of $\sigma_{g'}$.
	This contradicts the fact that $g'$ is $(L,m)$-aperiodic.
	
	\paragraph{}
	Recall that $g$ belongs to $Z \subset F$.
	The point $p$ is on $\sigma_g$.
	Therefore there is a point $h \in F$ which lies on $\sigma_g$ between 1 and $p$ such that $\dist hp \leq r$ (see Figure~\ref{fig: main lemma}).
	We put $u = h^{-1}l$ and $u'= w^{-m}l^{-1}g$.
	Hence $g=huw^mu'$.
	
	\begin{figure}[ht]
	\centering
		\includegraphics{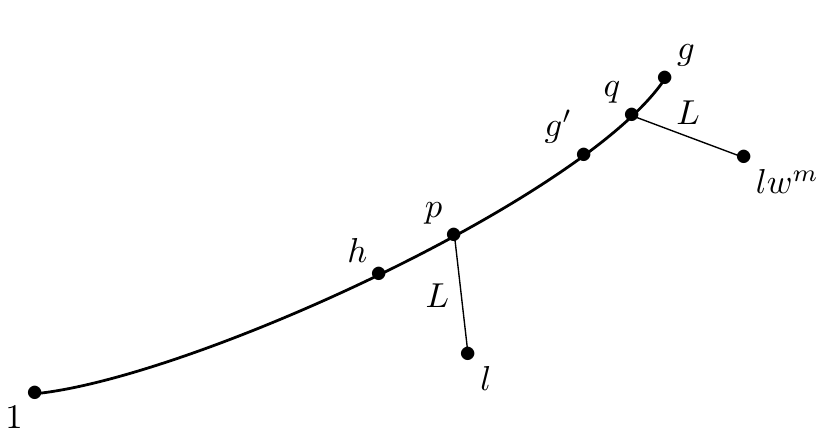}
	\caption{Positions of the points $h$, $p$, $g'$ and $q$.}
	\label{fig: main lemma}
	\end{figure}
	
	\paragraph{Claim 2:} $|g| \geq |h| + |w^m| - 2L$. 
	The points $p$ and $q$ lie on $\sigma_g$ between $g$ and $h$.
	Thus $|g| \geq |h| + \dist pq$.
	Recall that $p$ and $q$ are the respective projections of $l$ and $lw^m$ on $\sigma_g$.
	In particular $\dist lp, \dist {lw^m}q \leq L$.
	It follows from the triangle inequality that $|g| \geq |h| + |w^m| - 2L$.
	
	\paragraph{Claim 3:} The isometry $h$ belongs to $F_{L,m}$.
	By construction $h$ belongs to $F$.
	It is sufficient to prove that $h$ is $(L,m)$-aperiodic.
	We assumed that $m\tau \geq 2L +r$.
	The previous inequality becomes
	\begin{displaymath}
		|g| \geq |h| + |w^m| - 2L \geq |h| + m\tau -2L \geq |h| + r.
	\end{displaymath}
	It follows that $\dist gh \geq r \geq \dist g{g'}$. 
	In other words $h$ belongs to $\sigma_{g'}$.
	Since $g'$ is $(L,m)$-aperiodic, so is $h$.

	\paragraph{Claim 4:} $|u|, |u'| \leq L + r$.
	By construction of $h$, 
	\begin{displaymath}
		|u| = \dist lh \leq \dist lp + \dist ph \leq L +r.
	\end{displaymath}
	On the other hand, using our first claim, 
	\begin{displaymath}
		|u'| = \dist g{lw^m} \leq \dist gq + \dist q{lw^m} \leq L + r.
	\end{displaymath}

	\paragraph{} Claims 2 to 4 exactly say that $g \in Z_w$, which concludes the proof.
\end{proof}

\begin{coro}
\label{res: coro Z-Zw subset F_m}
	For all $s \geq 0$, 
	\begin{displaymath}
		\card{F_{L,m}\cap \ball s} \geq \card{Z\cap \ball s} - \sum_{w\in C} \card{Z_w\cap \ball s}.
	\end{displaymath}
\end{coro}

\begin{lemm}
\label{res: minoration card Z}
	For all $s\geq 0$,
	\begin{math}
		\card{Z\cap \ball{s+r}} \geq \rho\nu \lambda^r\card{F_{L,m}\cap \ball s}.
	\end{math}
\end{lemm}

\rem Recall that $\rho$ and $\nu$ are respectively given by Lemma~\ref{res: separate elements} and Proposition~\ref{res: all properties F} whereas $r$ is the radius that we fixed at the begin of this section.

\begin{proof}
	Applying Lemma~\ref{res: separate elements} there exists a subset $P$ of $F_{L,m}\cap \ball {s}$ such that $\card P \geq \rho \card{F_{L,m}\cap\ball s}$ whereas for all distinct $h,h' \in P$, $\dist h{h'} >20 \delta$.
	By definition of $Z$,
	\begin{displaymath}
		\bigcup_{h \in P}h\left(\fantomB L_h\cap h^{-1}F\cap \ann r{9\delta}\right) \subset Z\cap \ball{s+r}.
	\end{displaymath}
	We claim that this union is in fact a disjoint union.
	Assume on the contrary that this assertion is false.
	In particular, there are two distinct elements $h,h' \in P$ and $u,u' \in \ann r{9\delta}$ such that $h$ and $h'$ lie on the lexicographic geodesic from 1 to $hu=h'u'$.
	Since $u$ and $u'$ belongs to $\ann r{9\delta}$, $\dist h{h'}=\dist{|h|}{|h'|}\leq 9 \delta$.
	However $h,h' \in P$, thus $h=h'$.
	Contradiction.
	Therefore we have
	\begin{displaymath}
		\card{Z\cap \ball{s+r}} \geq \sum_{h \in P}\card{L_h\cap h^{-1}F\cap \ann r{9\delta}}.
	\end{displaymath}
	Nevertheless $P \subset F$.
	It follows from Proposition~\ref{res: all properties F}, that for all $h \in P$, $\card{L_h\cap h^{-1}F\cap \ann r{9\delta}} \geq \nu \lambda^r$.
	Consequently
	\begin{displaymath}
		\card{Z\cap \ball{s+r}}  \geq \nu \lambda^r \card P \geq  \rho\nu \lambda^r \card{F_{L,m}\cap \ball s}.
	\end{displaymath}
\end{proof}

\begin{lemm}
\label{res: majoration Z_w}
	For all $w \in G\setminus\{1\}$, for all $s \geq 0$, 
	\begin{displaymath}
		\card{Z_w\cap \ball{s+r}} \leq \alpha^2 \lambda^{2(L+r)}\card{F_{L,m}\cap \ball {s+r + 2L - m\len[stable]w}}.
	\end{displaymath}
\end{lemm}

\begin{proof}
	Let $g$ be an element of $Z_w\cap \ball{s+r}$.
	By definition there are $h \in F_{L,m}$ and $u,u' \in \ball{L +r}$ such that $g=huw^mu'$. 
	Moreover $|g| \geq |h| + |w^m| -2 L \geq |h| + m \len[stable] w - 2 L$.
	Consequently $h$ belongs to $\ball{s+r + 2L - m\len[stable]w}$.
	Hence $Z_w\cap \ball{s+r}$ is a subset of
	\begin{displaymath}
		\left(\fantomB F_{L,m}\cap \ball {s+r + 2L - m\len[stable]w}\right)\ball{L +r}w^m\ball{L +r}.
	\end{displaymath}
	The conclusion follows from Proposition~\ref{res:size of balls}.
\end{proof}

Let us summarize.
Let $s\in \R_+$.
By Corollary~\ref{res: coro Z-Zw subset F_m} and Lemma~\ref{res: minoration card Z}, 
\begin{eqnarray}
	\card{F_{L,m}\cap \ball {s+r}}
	& \geq & \card{Z\cap \ball {s+r}} - \sum_{w\in C} \card{Z_w\cap \ball {s+r}}\\
	& \geq & \rho\nu \lambda^r\card{F_{L,m}\cap \ball s} - \sum_{w\in C} \card{Z_w\cap \ball {s+r}}.
	\label{eqn: min F_m with quantifiers 1}
\end{eqnarray}
Let $j \in \N^*$.
We consider $w \in C$ such that $\len[stable] w = j\tau$ (see Theorem~\ref{res: stable translation length discrete}).
By Lemma~\ref{res: majoration Z_w},
\begin{displaymath}
	 \card{Z_w\cap \ball {s+r}}
	 \leq \alpha^2 \lambda^{2(L + r)} \card{F_{L,m}\cap \ball {s+r+ 2L - mj\tau}}.
\end{displaymath}
Since $w$ is cyclically reduced it satisfies $|w| = \len w \leq \len[stable] w +50 \delta$.
Hence such an isometry belongs to $\ball{j\tau +50\delta}$.
The number of elements $w \in C$ such that $\len[stable]w = j\tau$ is therefore bounded above by $\card{\ball{j\tau+50\delta}}\leq \alpha \lambda^{j\tau+50\delta}$.
Consequently
\begin{displaymath}
	\sum_{w\in C} \card{Z_w\cap \ball {s+r}}
	\leq  \alpha^3 \lambda^{2(L + r+25\delta)} \sum_{j\geq 1} \lambda^{j\tau} \card{F_{L,m}\cap \ball {s+r+ 2L - mj\tau}}.
\end{displaymath}
Note that in the sum on the right-hand side all but finitely many terms vanish.
Combining this last inequality with~(\ref{eqn: min F_m with quantifiers 1}) we get
\begin{eqnarray*}
	\lefteqn{\card{F_{L,m}\cap \ball{s+r}} \geq} \\
	&  & \rho\nu \lambda^r \card{F_{L,m}\cap \ball s} -\alpha^3 \lambda^{2(L + r+25\delta)}  \sum_{j\geq 1} \lambda^{j\tau}\card{F_{L,m}\cap \ball{ s+r+2L-mj\tau}}.
\end{eqnarray*}
Finally we proved the following proposition.

\begin{prop}
\label{res: min F_m with quantifiers}
	There exist positive constants $\tau$, $\kappa_1$ and $\kappa_2$ which only depends on $G$ and $A$ with the following property.
	Let $L > 0$.
	Let $r$ be a real number larger than $10 \delta$.
	There exists a subset $F$ of $G$ containing 1 such that for all integers $m$ satisfying $m\tau >  2L +r$, for all $s \in \R_+$,
	\begin{eqnarray*}
		\lefteqn{\card{F_{L,m}\cap \ball{s+r}} \geq} \\
		&  & \kappa_1 \lambda^r \card{F_{L,m}\cap \ball s} - \kappa_2 \lambda^{2(L+ r)} \sum_{j\geq 1} \lambda^{j\tau}\card{F_{L,m}\cap \ball{ s+r+2L-mj\tau}}.
	\end{eqnarray*}
\end{prop}

Before proving Proposition~\ref{res: growth F_m}, we introduce a family of auxiliary maps.
For all $L \geq 0$, for all $m \in \N^*$, the function $f_{L,m} : \left(\sqrt[m]\lambda, \lambda \right) \times \R_+ \rightarrow \R$ is given by 
\begin{eqnarray*}
	f_{L,m}(\mu,r) 
	& = & \kappa_1 \lambda^r  - \kappa_2 \lambda^{4(L + r)} \sum_{j\geq 1} \left(\frac{\lambda}{\mu^m}\right)^{ j\tau}\\
	& = &  \kappa_1 \lambda^r  - \kappa_2 \lambda^{4(L + r)} \frac{\lambda^\tau}{\mu^{m\tau} -\lambda^\tau}
\end{eqnarray*}

\begin{prop}
\label{res: min map f}	
	Let $L \geq 0$.
	There exists $m_0 \in \N$ such that for all integers $m \geq m_0$,
	\begin{displaymath}
		f_{L,m}\left( \lambda \left(1- \frac am\right), mb \right) \geq \left[\lambda \left(1- \frac am\right)\right]^{mb}.
	\end{displaymath}
	where  $a = 5(1-\kappa_1)/\tau$ and $b = \tau/4$.
\end{prop}

\begin{proof}
	For every $m \in \N^*$ we put $\mu_m = \lambda\left(1- a/m\right)$.
	Note that if $m$ is sufficiently large $\mu_m \in \left(\sqrt[m]\lambda, \lambda \right)$.
	Moreover we have the following properties.
	\begin{eqnarray*}
		f_{L,m}\left( \lambda \left(1- \frac am\right), mb \right)  & = & \lambda^{mb} \left[\kappa_1 + \operatorname{o}\displaylimits_{m \rightarrow + \infty}(1) \right], \\
		(\mu_m)^{mb} & = & \lambda^{mb} \left[1-ab + \operatorname{o}\displaylimits_{m \rightarrow + \infty}(1) \right]
	\end{eqnarray*}
	Since $1-ab < \kappa_1$ there exists $m_0$ such that for every $m \geq m_0$
	\begin{displaymath}
		f_{L,m}\left( \lambda \left(1- \frac am\right), mb \right) \geq (\mu_m)^{mb}.
	\end{displaymath}
	Note that $m_0$ only depends $G$, $A$ and $L$.
\end{proof}

\begin{proof}[Proof of Proposition~\ref{res: growth F_m}]
	Let $L \geq 0$.
	We put $a = 5(1-\kappa_1)/\tau$ and $b = \tau/4$.
	According to Proposition~\ref{res: min map f}, $m_0 \in \N$ such that for all $m \geq m_0$.
	\begin{enumerate}
		\item $mb \geq 10 \delta$,
		\item $m\tau \geq 2L + mb$,
		\item $m\ln\left(\lambda\left(1 - \frac am\right)\right) > \ln \lambda$,
		\item $f_{L,m}\left( \lambda \left(1- \frac am\right), mb \right) \geq \left[\lambda \left(1- \frac am\right)\right]^{mb}$.	
	\end{enumerate}
	Let $m \geq m_0$.
	For simplicity of notation we write $\mu = \lambda\left(1- a/m\right)$ and $r = mb$.
	Hence the previous inequalities can be written $r\geq 10 \delta$, $m\tau \geq 2L+r$, $\lambda\mu^{-m} <1$ and  $f_{L,m}(\mu,r) \geq \mu^r$.	
	By Proposition~\ref{res: min F_m with quantifiers}, there exists a subset $F$ of $G$ containing 1 such that for all $s\geq 0$,
	\begin{eqnarray*}
		\lefteqn{\card{F_{L,m}\cap \ball{s+r}}}\\
		& \geq & \kappa_1\lambda^r \card{F_{L,m}\cap \ball s} -\kappa_2 \lambda^{2(L+ r)} \sum_{j\geq 1} \lambda^{j\tau}\card{F_{L,m}\cap \ball{s+r+2L-mj\tau}}.
	\end{eqnarray*}
	We now prove by induction that for all $i \in \N$,
	\begin{equation}
	\label{eqn: min map f - induction}
	\tag{$\mathcal H_i$}
		\card{F_{L,m}\cap \ball{ir}} \geq \mu^r \card{F_{L,m}\cap{\ball {(i-1)r}}}.
	\end{equation}
	$(\mathcal H_0)$ is obviously true.
	Assume the the induction hypotheses holds for every integer smaller or equal to $i$.
	In particular, for all $t \geq 0$ we have
	\begin{displaymath}
		\card{F_{L,m}\cap{\ball {ir-t}}}
		\leq \mu^{-\left\lfloor \frac tr \right\rfloor r}\card{F_{L,m}\cap\ball {ir}} 
		\leq \mu^{r-t}\card{F_{L,m}\cap{\ball {ir}}}.
	\end{displaymath}
	By construction of $m_0$, for all $j\geq 1$, we have $mj\tau - 2L - r  \geq 0$, thus
	\begin{eqnarray*}
		\lambda^{j\tau}\card{F_{L,m}\cap \ball{ir+r+2L-mj\tau}} 
		& \leq & \left(\frac \lambda{\mu^m}\right)^{j\tau} \mu^{2(r+L)}\card{F_{L,m}\cap\ball {ir}} \\
		& \leq & \left(\frac \lambda{\mu^m}\right)^{j\tau} \lambda^{2(r+L)} \card{F_{L,m}\cap\ball {ir}}.
	\end{eqnarray*}
	Note that $m_0$ has been chosen in such a way that $\lambda\mu^{-m} <1$.
	Hence by summing these inequalities we obtain,
	\begin{eqnarray*}
		\card{F_{L,m}\cap \ball{(i+1)r}}
		& \geq & \left[ \kappa_1\lambda^r  -\kappa_2 \lambda^{4(L+r)} \sum_{j\geq 1}\left(\frac \lambda{\mu^m}\right)^{j\tau} \right]  \card{F_{L,m}\cap\ball {ir}}\\
		& \geq & f_{L,m}(\mu, r)  \card{F_{L,m}\cap\ball {ir}}.
	\end{eqnarray*}
	However, by construction, $f_{L,m}(\mu,r) \geq \mu^r$.
	Consequently $(\mathcal H_{i+1})$ holds.
	A second induction shows that for all $i \in \N$, $\card{F_{L,m}\cap \ball{ir}}\geq \mu^{ir}$.
	Therefore the growth rate of $F_{L,m}$ is larger than $\mu =  \lambda \left(1- a/m\right)$.
\end{proof}

\section{Growth of periodic quotients}
\label{sec: periodic quotients}

The infiniteness of the free Burnside groups is a consequence of the following result.
We will use it to estimate the growth rate of $\burn kn$.

\begin{theo}[S.I.~Adian {\cite{Adi79}}]
\label{res: aperiodic one-to-one - free case}
	Let $k \geq 2$.
	Let $A$ be a free generating set of $\free k$.
	There exist integers $n_0$ and $\eta$ such that for every odd exponent $n \geq n_0$ the following holds.
	Let $w$ be a reduced word over $A \cup A^{-1}$. 
	If $w$ does not contain a subword of the form $u^{\lfloor n/2\rfloor - \eta}$ then $w$ represents a non-trivial element of $\burn kn$.
\end{theo}

\begin{theo}
	Let $k \geq 2$.
	Let $A$ be a free generating set of $\free k$.
	There exist $\kappa >0$ and an integer $n_0$ such that for every odd exponent $n \geq n_0$ the exponential growth rate of $\burn kn$ with respect to the image of $A$ is larger than
	\begin{displaymath}
		(2k-1)\left( 1 - \frac \kappa{(2k-1)^{n/2}} \right)
	\end{displaymath}
\end{theo}

\begin{proof}
	The constants $n_0$ and $\eta$ are the ones given by Theorem~\ref{res: aperiodic one-to-one - free case}.
	We fix $a > 2k$.
	According to Proposition~\ref{res: growth F_m - free case}, there exists an integer $m_0$  such that for every $m \geq m_0$ the exponential growth rate of the set of $m$-aperiodic words is larger than 
	\begin{displaymath}
		(2k-1)\left( 1 - \frac a{(2k-1)^m} \right).
	\end{displaymath}
	Let $n \geq \max\left\{n_0,2m_0 +2\eta+2\right\}$ be an odd integer.
	By Theorem~\ref{res: aperiodic one-to-one - free case}, the natural map $\free k \rightarrow \burn kn$ restricted to the set of $\left(\lfloor n/2 \rfloor- \eta\right)$-aperiodic elements is one-to-one.
	Therefore the growth rate of $\burn kn$ is larger than the one of this set. 
	In particular it is larger than 
	\begin{displaymath}
		(2k-1)\left( 1 - \frac {a(2k-1)^{\eta+1}}{(2k-1)^{n/2}} \right),
	\end{displaymath}
	which ends the proof.
\end{proof}

\paragraph{}
In \cite{Olc91}, A.Y.~Ol'shanski\u\i\ solved the Burnside problem for hyperbolic groups.

\begin{theo}[Ol'shanski\u\i\ {\cite{Olc91}}]
	Let $G$ be a non-elementary, torsion-free hyperbolic group.
	There exists an integer $n(G)$ such that for all odd exponents $n \geq n(G)$ the quotient $G/G^n$ is infinite.
\end{theo}

The proof relies on the following fact. 
If $n$ is large enough, then the restriction of the canonical projection $G \twoheadrightarrow G/G^n$ to a set of aperiodic elements (which is infinite) is injective.
More precisely he showed the following statement.

\begin{theo}[Ol'shanski\u\i\ {\cite{Olc91}}]
\label{res: aperiodic one-to-one}
	Let $G$ be a non-elementary, torsion-free hyperbolic group.
	There exist constants $L$, $\epsilon$ and $n(G)$ with the following property.
	Let $n \geq n(G)$ be an odd integer and $m \leq \epsilon n$.
	Then the restriction of $G \twoheadrightarrow G/G^n$ to the set of $(L, m)$-aperiodic elements is one-to-one.
\end{theo}

\rems The definition of aperiodic words used by A.Y.~Ol'shanski\u\i\ is slightly different from ours (see Section~\ref{sec:avoiding large powers}).
He says that an element $g \in G$ contains a $(L, m)$-power if there is $(p,w) \in G\times C$ such that both $p$ and $pw^n$ belong to the $L$-neighbourhood of some geodesic between 1 and $g$ (not necessary $\sigma_g$).
However in a hyperbolic space, two geodesics joining the same extremities are $4\delta$-close one from the other.
Hence a $(L, m)$-aperiodic element in the sense of Ol'shanski\u\i\ is $(L +4\delta, m)$-aperiodic in our sense and conversely.
Therefore the statement of Theorem~\ref{res: aperiodic one-to-one} with one or the other definition are equivalent.
An other approach based on the work of T.~Delzant and M.~Gromov \cite{DelGro08} can be found in \cite{Cou11b}.

\begin{theo}
	Let $G$ be a non-elementary, torsion-free hyperbolic group and $\lambda$ its exponential growth rate with respect to a finite generating set $A$ of $G$.
	There exists a positive number $\kappa$ such that for sufficiently large odd exponents $n$ the exponential growth rate of $G/G^n$ with respect to $A$ is larger than $\lambda \left(1-  \kappa/n\right)$.
\end{theo}

\begin{proof}
	The parameters $L$, $\epsilon$ and $n(G)$ are given by Theorem~\ref{res: aperiodic one-to-one}.
	The constants $a$ and $m_0$ (which only depend on $G$ and $L$) are then given by Proposition~\ref{res: growth F_m}.
	Let $n \geq \left\{\epsilon^{-1}( m_0 +1),  n(G) , 2\epsilon^{-1} \right\}$ be an odd integer.
	We put $m = \lfloor \epsilon n \rfloor$.
	According to Proposition~\ref{res: growth F_m}, there exists a subset $F$ of $G$ such that the exponential growth rate of $F_{L,m}$ is larger than $\lambda \left(1-  a/m\right) \geq \lambda \left(1-  {2a}/{\epsilon n}\right)$.
	By Theorem~\ref{res: aperiodic one-to-one} the restriction of $G \twoheadrightarrow G/G^n$ to $F_{L,m}$ is one-to-one.
	On the other hand for every $g \in G$ the length of $g$ with respect to $A$ is larger than the length of its image in $G/G^n$ with respect to the image of $A$.
	Therefore for all $r \geq 0$, the ball of radius $r$ in $G/G^n$ contains at least $\card{F_{L,m} \cap \ball r}$ elements.
	Thus the exponential growth rate of $G/G^n$ is larger than the one of $F_{L,m}$.
	In particular it is larger than $\lambda \left(1-  {2a}/{\epsilon n}\right)$.
\end{proof}

\makebiblio
%\todos

\vspace*{1cm}
\noindent %
R\'emi Coulon \\
Vanderbilt University \\
Department of Mathematics \\
remi.coulon@vanderbilt.edu \\


\begin{thebibliography}{10}

\bibitem{Adi79}
S.~I. Adian.
\newblock {\em {The Burnside problem and identities in groups}}, volume~95 of
  {\em Ergebnisse der Mathematik und ihrer Grenzgebiete [Results in Mathematics
  and Related Areas]}.
\newblock Springer-Verlag, Berlin, 1979.

\bibitem{ArzLys02}
G.~N. Arzhantseva and I.~G. Lysenok.
\newblock {Growth tightness for word hyperbolic groups}.
\newblock {\em Mathematische Zeitschrift}, 241(3):597--611, 2002.

\bibitem{Atabekyan:2009ih}
V.~S. Atabekyan.
\newblock {Uniform nonamenability of subgroups of free Burnside groups of odd
  period}.
\newblock {\em Rossi\u\i skaya Akademiya Nauk. Matematicheskie Zametki},
  85(4):516--523, 2009.

\bibitem{Bur02}
W.~Burnside.
\newblock {On an unsettled question in the theory of discontinuous groups}.
\newblock {\em Quart.J.Math.}, 33:230--238, 1902.

\bibitem{Can84}
J.~W. Cannon.
\newblock {The combinatorial structure of cocompact discrete hyperbolic
  groups}.
\newblock {\em Geometriae Dedicata}, 16(2):123--148, 1984.

\bibitem{Coo93}
M.~Coornaert.
\newblock {Mesures de Patterson-Sullivan sur le bord d'un espace hyperbolique
  au sens de Gromov}.
\newblock {\em Pacific Journal of Mathematics}, 159(2):241--270, 1993.

\bibitem{CooDelPap90}
M.~Coornaert, T.~Delzant, and A.~Papadopoulos.
\newblock {\em {G\'eom\'etrie et th\'eorie des groupes}}, volume 1441 of {\em
  Lecture Notes in Mathematics}.
\newblock Springer-Verlag, Berlin, 1990.

\bibitem{Cou11b}
R.~Coulon.
\newblock {Detecting trivial elements of Burnside groups}, \\
  http://www.math.vanderbilt.edu/$\sim$coulonrb/papiers/algo.pdf.
\newblock 2012.

\bibitem{Del96}
T.~Delzant.
\newblock {Sous-groupes distingu{\'e}s et quotients des groupes hyperboliques}.
\newblock {\em Duke Mathematical Journal}, 83(3):661--682, 1996.

\bibitem{DelGro08}
T.~Delzant and M.~Gromov.
\newblock {Courbure m{\'e}soscopique et th{\'e}orie de la toute petite
  simplification}.
\newblock {\em Journal of Topology}, 1(4):804--836, 2008.

\bibitem{GhyHar90}
{\'E}.~Ghys and P.~de~la Harpe.
\newblock {\em {Sur les groupes hyperboliques d'apr{\`e}s Mikhael Gromov}},
  volume~83 of {\em Progress in Mathematics}.
\newblock Birkh{\"a}user Boston Inc., Boston, MA, 1990.

\bibitem{Gro87}
M.~Gromov.
\newblock {Hyperbolic groups}.
\newblock In {\em Essays in group theory}, pages 75--263. Springer, New York,
  1987.

\bibitem{Hal57}
M.~Hall~Jr.
\newblock {Solution of the Burnside problem of exponent $6$}.
\newblock {\em Proceedings of the National Academy of Sciences of the United
  States of America}, 43:751--753, 1957.

\bibitem{Iva94}
S.~V. Ivanov.
\newblock {The free Burnside groups of sufficiently large exponents}.
\newblock {\em International Journal of Algebra and Computation},
  4(1-2):ii+308, 1994.

\bibitem{Kou98}
M.~Koubi.
\newblock {Croissance uniforme dans les groupes hyperboliques}.
\newblock {\em Universit{\'e} de Grenoble. Annales de l'Institut Fourier},
  48(5):1441--1453, 1998.

\bibitem{LevWae33}
F.~Levi and B.~van~der Waerden.
\newblock {\"Uber eine besondere Klasse von Gruppen}.
\newblock {\em Abhandlungen aus dem Mathematischen Seminar der Universit\"at
  Hamburg}, 9(1):154--158, Dec. 1933.

\bibitem{Lys96}
I.~G. Lysenok.
\newblock {Infinite Burnside groups of even period}.
\newblock {\em Izvestiya Akademii Nauk SSSR. Seriya Matematicheskaya},
  60(3):3--224, 1996.

\bibitem{NovAdj68c}
P.~S. Novikov and S.~I. Adian.
\newblock {Infinite periodic groups.}
\newblock {\em Izvestiya Akademii Nauk SSSR. Seriya Matematicheskaya}, 32,
  1968.

\bibitem{Olc91}
A.~Y. Ol'shanski{\u \i}.
\newblock {Periodic quotient groups of hyperbolic groups}.
\newblock {\em Matematicheski\u\i\ Sbornik}, 182(4):543--567, 1991.

\bibitem{Osi07}
D.~V. Osin.
\newblock {Uniform non-amenability of free Burnside groups}.
\newblock {\em Archiv der Mathematik}, 88(5):403--412, 2007.

\bibitem{San40}
I.~N. Sanov.
\newblock {Solution of Burnside's problem for exponent 4}.
\newblock {\em Leningrad State Univ. Annals [Uchenye Zapiski] Math. Ser.},
  10:166--170, 1940.

\end{thebibliography}
\end{document}